\documentclass[a4paper,11pt]{article}
\usepackage{amsmath}
\usepackage{amssymb}
\usepackage{theorem}
\usepackage{pstricks}
\usepackage{euscript}
\usepackage{epic,eepic}
\usepackage{graphicx}
\PassOptionsToPackage{normalem}{ulem}
\usepackage{amsmath}
\usepackage{amssymb}
\usepackage{theorem}
\usepackage{pstricks}
\usepackage{euscript}
\usepackage{epic,eepic}
\usepackage{graphicx}
\usepackage{setspace}
\PassOptionsToPackage{normalem}{ulem}

\newcommand{\menge}[2]{\big\{{#1} \;|\; {#2}\big\}}

\newcommand{\emp}{\ensuremath{{\varnothing}}}

\newcommand{\scal}[2]{\left\langle{#1}\mid {#2} \right\rangle}

\newcommand{\vuo}{\ensuremath{\mbox{\footnotesize$\square$}}}

\newcommand{\HH}{\ensuremath{\mathcal H}}
\newcommand{\GG}{\ensuremath{\mathcal G}}

\newcommand{\KKK}{\ensuremath{\boldsymbol{\mathcal K}}}
\newcommand{\GGG}{\ensuremath{\boldsymbol{\mathcal G}}}

\newcommand{\RR}{\ensuremath{\mathbb R}}
\newcommand{\KK}{\ensuremath{\mathcal K}}

\newcommand{\NN}{\ensuremath{\mathbb N}}

\newcommand{\dom}{\ensuremath{\operatorname{dom}}}

\newcommand{\prox}{\ensuremath{\operatorname{prox}}}

\newcommand{\argmin}{\ensuremath{\operatorname{argmin}}}
\newcommand{\ran}{\ensuremath{\operatorname{ran}}}
\newcommand{\zer}{\ensuremath{\operatorname{zer}}}
\newcommand{\gra}{\ensuremath{\operatorname{gra}}}

\newcommand{\vv}{\ensuremath{\boldsymbol{v}}}

\newcommand{\xx}{\ensuremath{\boldsymbol{x}}}

\newcommand{\yy}{\ensuremath{\boldsymbol{y}}}

\newcommand{\ee}{\ensuremath{\boldsymbol{e}}}

\newcommand{\cc}{\ensuremath{\boldsymbol{c}}}
\newcommand{\dd}{\ensuremath{\boldsymbol{d}}}
\newcommand{\aaa}{\ensuremath{\boldsymbol{a}}}
\newcommand{\ww}{\ensuremath{\boldsymbol{w}}}
\newcommand{\BB}{\ensuremath{\boldsymbol{B}}}

\newcommand{\AAA}{\ensuremath{\boldsymbol{A}}}
\newcommand{\BBB}{\ensuremath{\boldsymbol{B}}}

\newcommand{\SSS}{\ensuremath{\boldsymbol{S}}}

\newcommand{\Id}{\ensuremath{\operatorname{Id}}}

\newcommand{\weakly}{\ensuremath{\rightharpoonup}}

\newtheorem{theorem}{Theorem}[section]
\newtheorem{lemma}[theorem]{Lemma}

\newtheorem{corollary}[theorem]{Corollary}

\newtheorem{definition}[theorem]{Definition}
\theoremstyle{plain}{\theorembodyfont{\rmfamily}
}
\theoremstyle{plain}{\theorembodyfont{\rmfamily}
\newtheorem{condition}[theorem]{Condition}}
\theoremstyle{plain}{\theorembodyfont{\rmfamily}
\newtheorem{algorithm}[theorem]{Algorithm}}
\theoremstyle{plain}{\theorembodyfont{\rmfamily}
\newtheorem{example}[theorem]{Example}}
\theoremstyle{plain}{\theorembodyfont{\rmfamily}
\newtheorem{problem}[theorem]{Problem}}
\theoremstyle{plain}{\theorembodyfont{\rmfamily}
\newtheorem{remark}[theorem]{Remark}}
\theoremstyle{plain}{\theorembodyfont{\rmfamily}
\newtheorem{notation}[theorem]{Notation}}

\definecolor{labelkey}{rgb}{0,0.08,0.45}
\definecolor{refkey}{rgb}{0,0.6,0.0}
\definecolor{Brown}{rgb}{0.45,0.0,0.05}
\definecolor{dgreen}{rgb}{0.00,0.49,0.00}
\definecolor{dblue}{rgb}{0,0.08,0.75}
\numberwithin{equation}{section}
\renewcommand\theenumi{(\roman{enumi})}

\title{ A primal-dual backward reflected forward splitting algorithm for structured monotone inclusions }
\author{Vu Cong Bang$^1$, Dimitri Papadimitriou$^1$ and
Vu Xuan Nham$^2$  \\[5mm]
$^1$ Belgium Research Center (BeRC), Huawei\\
             3001 Leuven, Belgium\\
 $^2$ University of Transport Technology, Hanoi, Vietnam  \\       
 bangcvvn@gmail.com;dpapadimitriou@3nlab.org; nhamvx@utt.edu.vn} 
\date{}
\begin{document}
\maketitle
\begin{abstract} 
We propose a primal-dual backward reflected forward splitting method  for solving  structured primal-dual monotone inclusion in real Hilbert space. 
The algorithm allows to use the inexact computations of the Lipschitzian and cocoercive operators. 
The strong convergence of the generated iterative sequence is proved under the strong monotonicity condition, whilst the weak convergence is formally proved under several conditioned used in the literature. An application to a structured minimization problem is supported.
\end{abstract}


\section{Introduction}
Primal-dual splitting methods  have been investigated extensively in 
the literature and found many applications in applied mathematics  \cite{Bot2013,Luis15,Buicomb,siop2, plc6,Combettes13,Icip14,Condat2013,sva2,te15,Latafat2018,Pesquet15, Bang13,Bang14,Vu2013}.
The first primal-dual splitting method was proposed in \cite{siop2} 
for the sum of a linearly composed operator and a maximally monotone operator. 
This framework was then extended to both  monotone inclusion and system monotone inclusions with mixtures of composite, Lipschitzian, 
and parallel-sum type monotone operators in \cite{plc6} and \cite{Combettes13,Bang14}, respectively. Further  developments and convergence analysis 
of the framework in \cite{plc6} can be found in \cite{Bot2014}. These methods can be viewed as primal-dual monotone+skew splitting and they have the structure of the forward-backward-forward splitting \cite{Tseng00,siop2}.
When the monotone and Lipschitzian operators in \cite{plc6} are 
restricted to be cocoercive operators, an alternative primal-dual  splitting methods were proposed in \cite{Bang13} and \cite{Icip14} which
have the structure of forward-backward splitting \cite{Lions1979,plc4}. A Douglas-Rachford type primal-dual method for solving monotone inclusions with mixtures of composite and parallel-sum type monotone operators was also proposed in \cite{Bot2013}.
Very recently, these frameworks have been unified into a highly structured multivariate  system of monotone inclusions involving a mix of set-valued, cocoercive, and Lipschitzian monotone operators \cite{Buicomb} where the asynchronous block-iterative outer approximation method is proposed.

 In this paper,  we revisit the following problem which is slightly extension of the framework in \cite{plc6} and it is an instance of the one in \cite{Buicomb}.  We further exploit its duality nature
 by incorporating  the structure of the backward-forward  \cite{Attouch18} and the structure  of forward reflected backward \cite{Malitsky18b}.
 This will lead to a new type of primal-dual splitting method. 
 \begin{problem}
\label{p:primal-dual}
Let $m$ be a strictly positive integer, 
let $(\beta_i)_{0\leq i\leq m}$ and $(\mu_i)_{0\leq i\leq m}$ be in $\left]0,+\infty\right[^{m+1}$.
Let $B\colon\HH\to\HH$ be  $\beta_0$-cocoercive, 
let  $A\colon \HH\to 2^{\HH}$ be maximally monotone, let $Q\colon\HH\to\HH$ be monotone and
 $\mu_0$-Lipschitzian, and let $z\in\HH$.
Let $(\GG_i)_{1\leq i\leq m}$ be real Hilbert spaces. 
For every $i\in\{1,\ldots,m\}$, let $A_i\colon\GG_i\to 2^{\GG_i}$ be  maximally monotone, 
 let $B_i\colon\GG_i\to 2^{\GG_i}$ be   $\beta_i$-cocoercive, let $Q_i\colon\GG_i\to \GG_i$ be monotone 
and $\mu_i$-Lipschitzian, let $r_i\in\GG$, and
 let $L_i\colon\HH\to \GG_i$ be a bounded linear operator such that $0\not=\sum_{i=1}^m\|L_i\|^2$. 
Suppose that 
\begin{equation}
\label{c:1}
z \in \ran\bigg(A+ \sum_{i=1}^m L_{i}^* (A_{i}+ B_i+Q_i)^{-1}( L_i\cdot-r_i) + B +Q \bigg).
\end{equation}
The primal inclusion is to find $\overline{x}$ such that
\begin{equation}
\label{e:pri}
z \in A\overline{x}+ \sum_{i=1}^m L_{i}^* \big(A_{i}+ B_i+Q_i\big)^{-1}  (L_i \overline{x}-r_i)+ B\overline{x}+Q\overline{x},
\end{equation}
and the dual inclusion is to find $(\overline{v}_i)_{1\leq i\leq m} \in (\GG_i)_{1\leq i\leq m}$ such that 
\begin{equation}
\label{e:dual}
\big(\forall i\in\{1,\ldots,m\}\big) \; -r_i \in -L_i(A+B +Q)^{-1}\bigg(z-\sum_{i=1}^m L_{i}^*\overline{v}_i\bigg) + A_{i} \overline{v}_i+ B_{i} \overline{v}_i 
+Q_{i} \overline{v}_i.
\end{equation}
\end{problem}
When every Lipschitzian operator $Q$ and 
$(Q_i)_{1\leq i\leq m}$ is zero, Problem \ref{p:primal-dual} reduces to the one in \cite{Vu2013}. When every cocoercive operator $B$ and 
$(B_i)_{1\leq i\leq m}$ is zero,   Problem \ref{p:primal-dual} reduces to the problem-dual problem in \cite{plc6}. Further special cases of this primal-dual framework can be found in \cite{plc6,Bot2013,Vu2013,Pesquet15} as well as in the recent work \cite{Buicomb}.

In Section \ref{s:background}, we recall some basic notions in convex analysis and monotone operator theory. The proposed algorithm as well as convergence results are presented in Section \ref{s:Algorithm}. We provide an application to a structured minimization problem in Section \ref{s:Application}.

\section{Notation and Background}
\label{s:background}
We recall the standard notation in convex analysis and monotone operator theory  in \cite{livre1}.
 The scalar products and the  associated 
norms of all Hilbert spaces used in this paper are denoted respectively by 
$\scal{\cdot}{\cdot}$ and $\|\cdot\|$. 
The symbols $\weakly $ and $\to$ denote respectively 
weak and strong convergence.
Let $A\colon\HH\to 2^{\HH}$ be a set-valued operator. The domain of $A$ is denoted by $\dom(A)$ that 
is a set of all $x\in\HH$ such that $Ax\not= \emp$. The range of $A$ is $\ran(A) = \menge{u\in\HH}{(\exists x\in\HH) u\in Ax }$. 
The graph of $A$ is 
$\gra(A) = \menge{(x,u)\in\HH\times\HH}{u\in Ax}$. 
The inverse of $A$ 
is $A^{-1}\colon u \mapsto \menge{x}{ u\in Ax}$. 
The zero set of $A$ is $\zer(A) = A^{-1}0$.
\begin{definition}  We have the following definitions:
\begin{enumerate}
\item We say that $A\colon\HH\to 2^\HH$ is  monotone if 
\begin{equation}\label{oioi}
\big(\forall (x,u)\in \gra A\big)\big(\forall (y,v)\in\gra A\big)
\quad\scal{x-y}{u-v}\geq 0.
\end{equation}
\item We say that $A\colon\HH\to 2^\HH$ is maximally monotone if it is monotone and  there exists no monotone operator $B$ such that $\gra(B)$ properly contains $\gra(A)$.
\item We say that  $A\colon\HH\to 2^\HH$  is uniformly monotone at $x\in\dom (A)$, if there exists an   increasing function $\phi_A\colon\left[0,\infty\right[\to \left[0,\infty\right]$ that vanishes only at $0$ such that 
 \begin{equation}\label{unif}
\big(\forall u\in Ax\big)\big(\forall (y,v)\in\gra A\big)
\quad\scal{x-y}{u-v}\geq \phi_A(\|y-x\|).
\end{equation}
Alternatively, $A$ is called $\phi_A$-uniformly monotone with mudulus function $\phi_A$. 
\item The resolvent of $A$ and the Yosida approximation of $A$ of index $\gamma\in\RR$ are respectively defined by
\begin{equation}
 J_A=(\Id + A)^{-1}\;  \text{and} \; A_{\gamma} = (\gamma \Id +A^{-1})^{-1},
\end{equation}
where $\Id$ denotes the identity operator on $\HH$.  
\item A single-valued operator $B\colon\HH\to\HH$ is $\beta$-cocoercive, for some $\beta\in \left]0,+\infty\right[$, if
\begin{equation}
(\forall (x,y)\in\HH^2)\; \scal{x-y}{Bx-By} \geq \beta\|Bx-By\|^2.
\end{equation}  
If $\beta=1$, then $B$ is a firmly non-expansive operator.
\end{enumerate}
\end{definition}
The class of all lower semicontinuous convex functions 
$f\colon\HH\to\left]-\infty,+\infty\right]$ such 
that $\dom f=\menge{x\in\HH}{f(x) < +\infty}\neq\emp$ 
is denoted by $\Gamma_0(\HH)$. 
\begin{definition}  Let $f\in\Gamma_0(\HH)$.  We  recall the following notions in convex analysis.
\begin{enumerate}
\item 
The conjugate of $f$ is the function $f^*\in\Gamma_0(\HH)$ defined by
$f^*\colon u\mapsto
\sup_{x\in\HH}(\scal{x}{u} - f(x))$,
\item The subdifferential 
of $f$ is the maximally monotone operator 
\begin{equation}
 \partial f\colon\HH\to 2^{\HH}\colon x
\mapsto\menge{u\in\HH}{(\forall y\in\HH)\quad
\scal{y-x}{u} + f(x) \leq f(y)}
\end{equation} 
with inverse given by
\begin{equation}
(\partial f)^{-1}=\partial f^*.
\end{equation}
\item
The proximity operator of $f$ is
\begin{equation}
\label{e:prox}
\prox_f=J_{\partial f} \colon\HH\to\HH\colon x
\mapsto\underset{y\in\HH}{\argmin}\: f(y) + \frac12\|x-y\|^2.
\end{equation}
\item
 The infimal convolution of the two functions $\ell$ and $g$ from $\HH$ to $\left]-\infty,+\infty\right]$ is 
\begin{equation}
 \ell\;\vuo\; g\colon x \mapsto \inf_{y\in\HH}(\ell(y)+g(x-y)).
\end{equation}
\end{enumerate}
\end{definition}
\begin{notation}\label{n:main} In the setting of Problem \ref{p:primal-dual},
we denote $\KKK = \HH\oplus\GG_1\oplus\ldots\oplus\GG_m$  the Hilbert direct sum of the Hilbert spaces $\HH$ and $(\GG_i)_{1\leq i\leq m}$, where the scalar product and the  associated norm of $\KK$ are respectively defined as, for any  $(x,\vv) = (x,v_1,\ldots, v_m) \in \KKK$ and $(y,\ww) = (y,w_1,\ldots, w_m) \in \KKK$,
 \begin{equation}
\big((x,\vv), (y,\ww)\big) \mapsto \scal{x}{y} + \sum_{i=1}^m\scal{v_i}{w_i},
 \end{equation}
 and 
 \begin{equation}
 (x,\vv) \mapsto \sqrt{\|x\|^2 +\sum_{i=1}^m\|v_i\|^2}.
 \end{equation}
 The space $\GGG= \GG_1\oplus\ldots\oplus\GG_m$ is defined by the same fashion. We also use the following operators:
 \begin{equation}
 \begin{cases}
\SSS \colon \KKK\to\KKK \colon (x,v_1,\ldots, v_m) \mapsto  (Qx+\sum_{i=1}^m L_{i}^*v_i, Q_1v_1 -L_1x, \ldots, Q_mv_m -L_m x ),\\
 \BB \colon \KKK\to\KKK \colon (x,v_1,\ldots, v_m) \mapsto (Bx, B_{1}v_1,\ldots,  B_{m}v_m),\\
 \AAA\colon \KKK\to 2^{\KKK}\colon (x,v_1,\ldots, v_m) \mapsto (-z+Ax, r_1+ A_{1}v_1, \ldots, r_m+A_mv_m).
 \end{cases}
 \end{equation}
\end{notation}

\begin{lemma}{\rm \cite[Eqs.(3.12),(3.13),(3.21),(3.22)]{plc6}} \label{l:2}
 Suppose that the condition \eqref{c:1} is satisfied. Then, $\zer(\AAA+\SSS+\BB)\not=\emp$. Furthermore, 
  \begin{equation}\label{e:dss}
  (\overline{x}, \overline{v}_1,\ldots, \overline{v}_m) \in\zer(\AAA+\SSS+\BB) \Rightarrow \overline{x}\; \text{solves}\; \eqref{e:pri}\; \text{and}\;   ( \overline{v}_1,\ldots, \overline{v}_m)\;
   \text{solves}\; \eqref{e:dual}.
  \end{equation} 
\end{lemma}
\begin{lemma} {\rm \cite[Proposition 2.4(ii)\&(iii)]{Attouch18}} \label{l:1} The operator 
\begin{equation}
\Id-\gamma \SSS-\gamma\BB \colon \zer(\AAA+\BB+\SSS)\to \zer((\BB+\SSS)_{-\gamma} +\AAA_{\gamma}) 
\end{equation}
is a bijection with inverse $J_{\gamma \AAA}$. Furthermore, 
\begin{equation}
 \zer((\BB+\SSS)_{-\gamma} +\AAA_{\gamma})  = \zer( (\BB+\SSS) J_{\gamma \AAA} +\AAA_{\gamma}).
\end{equation}
\end{lemma}
The convergence analysis of this paper is relied on  the forward-reflected-backward in \cite{Malitsky18b}.
\begin{lemma}\label{l:M-T}
 Let $\beta$ and $\mu$ be in $\left]0,+\infty\right[$.
Let $B\colon\HH\to\HH$ be  $\beta$-cocoercive, 
let  $A\colon \HH\to 2^{\HH}$ be maximally monotone, let $Q\colon\HH\to\HH$ be monotone and
 $\mu$-Lipschitzian such that $\zer(A+B+Q) \not=\emp$. Let 
 $\gamma \in\left]0,+\infty\right[$ such that $\gamma < 2\beta/(1+4\beta\mu)$. 
 Let $x_{-1} $ and $x_0$ be in $\HH$. Set $y_{-1} = J_{\gamma A}x_{-1}$ and iterate 
 \begin{equation}
\label{algo:primafbl}
(\forall n\in\NN)\quad
\begin{array}{l}
\left\lfloor
\begin{array}{l}
y_n = J_{\gamma A} x_n\\
x_{n+1} = y_{n} - \gamma (2Q y_{n}- Qy_{n-1} )-\gamma By_{n}.
\end{array}
\right.\\[2mm]
\end{array}
\end{equation}
Then, the sequence $(y_n)_{n\in\NN}$ converges weakly to a point $\overline{x} \in zer(A+B+Q)$. 
 \end{lemma}
 
\section{Main results}
\label{s:Algorithm}
In view of Lemma \ref{l:2}, to solve Problem \ref{p:primal-dual}, it is natural to consider the problem of finding a point
\begin{equation}\label{e:3operators}
 (\overline{x}, \overline{v}_1,\ldots, \overline{v}_m) \in\zer(\AAA+\SSS+\BB),
 \end{equation}
 where $\AAA$, $\SSS$ and $\BB$ are defined by Notation \ref{n:main}. 
 We investigate the weak and strong convergence of Algorithm \ref{algo:main}, 
 which shares the same structure as
  the algorithm \eqref{algo:primafbl} applied to
the problem \eqref{e:3operators}. However, at each iteration, the algorithm uses the approximation of the single valued operators $B,Q,  (B_{i})_{1\leq i\leq m}$ and $(Q_{i})_{1\leq i\leq m}$.
\begin{algorithm} \label{algo:main}
For every $n\in\NN$,  let $B_n\colon \HH\to\HH$ and $Q_n\colon\HH\to \HH$.
For every $i\in\{1,\ldots m\}$,
 let $B_{i,n}\colon \GG_i\to\GG_i$ and $Q_{i,n}\colon\GG_i\to \GG_i$.
Let $\gamma$ be in $\left]0,\infty\right[$,
let $x_{-1}\in\HH$, $y_{-1} = J_{\gamma A}x_{-1} $, 
and,  for every $i\in\{1,\ldots,m\}$, let $v_{i,-1} \in\GG_i$ and $w_{i,-1}= J_{\gamma A}v_{i,-1}$.
Let $x_{0}\in\HH$ and,  for every $i\in\{1,\ldots,m\}$, let $v_{i,0} \in\GG_i$. Iterate 
\begin{equation}
\label{algo:primal-dual}
(\forall n\in\NN)\quad
\begin{array}{l}
\left\lfloor
\begin{array}{l}
y_n = J_{\gamma  A} (x_n+\gamma  z)\\
\operatorname{For \; \text{$i =1,\ldots,m$}}\\
\begin{array}{l}
\left\lfloor
\begin{array}{l}
w_{i,n} = J_{\gamma  A_{i}} (v_{i,n}-\gamma  r_i)\\
\end{array}
\right.\\[2mm]
\end{array}\\
x_{n+1} = y_{n}- \gamma \sum_{i=1}^m L_{i}^*(2w_{i,n}-w_{i,n-1}) - \gamma (2Q_n y_{n}- Q_ny_{n-1} )-\gamma B_ny_{n}\\
\operatorname{For \; \text{$i =1,\ldots,m$}}\\

\begin{array}{l}
\left\lfloor
\begin{array}{l}
v_{i,n+1} = w_{i,n} -\gamma (2Q_{i,n} w_{i,n}- Q_{i,n}w_{i,n-1} ) + \gamma L_i(2y_{n}-y_{n-1}) -\gamma B_{i,n} w_{i,n}.\\
\end{array}
\right.\\[2mm]
\end{array}\\
\end{array}
\right.\\[2mm]
\end{array}
\end{equation}
\end{algorithm} 
The convergence of the proposed algorithm relies on the following additional conditions which was introduced in \cite{Attouch10}.
\begin{condition} \label{c:new}
Assume that the operators $ B_n,Q_n, (B_{i,n})_{1\leq i\leq m}$ and $(Q_{i,n})_{1\leq i\leq m}$ verify the following conditions:
\begin{enumerate}
\item 
The operators $B_n-B$ and $Q_n-Q$ are $\kappa_{0,n}$-Lipschitz continuous with $\sum_{n\in\NN}\kappa_{0,n} <+\infty$ and 
\begin{equation}
    (\exists (c_0, d_0))\in \HH^2)(\forall n\in\NN) \; B_nc_0= Bc_0\; \text{and}\; Q_nd_0= Qd_0.  
\end{equation}
\item For every $i\in \{1,\ldots,m\}$, 
the operators $B_{i,n}-B_i$ and $Q_{i,n}-Q_i$ are $\kappa_{i,n}$-Lipschitz continuous with $\sum_{n\in\NN}\kappa_{i,n} <+\infty$, and 
\begin{equation}
    (\forall i\in\{1,\ldots,m\})(\exists (c_{0}, d_0))\in \GG_{i}^2)(\forall n\in\NN) \; B_{i,n}c_{i}= B_ic_i\; \text{and}\; Q_{i,n}d_i= Q_id_i.  
\end{equation}
\end{enumerate}
\end{condition}
\noindent  The main result of this paper are stated in the following theorems in which 
the weak convergence of $(y_n)_{n\in\NN}$ and  $(w_{1,n},\ldots, w_{m,n})_{n\in\NN}$ to 
 solutions to \eqref{e:pri} and \eqref{e:dual}, respectively, are obtained. 
Moreover, the strong convergence is also proved under the uniform monotonicity condition.
\begin{theorem} \label{t:1} 
Under the setting of Problem \ref{p:primal-dual} as well as 
Condition \ref{c:new},
let $\beta$, $\mu$ and $(\kappa_n)_{n\in\NN}$ be such that
\begin{equation}\label{eq:310}
0 < \beta <  \beta'=\min_{0\leq i\leq m} \beta_i\; \text{and}\; 
\mu =\sqrt{\sum_{i=1}^m\|L_i\|^2}+ \max_{0\leq i\leq m} \mu_i, (\forall n\in\NN)\; \kappa_n=\sqrt{\sum_{i=0}^m\kappa_{i,n}^2 },
\end{equation}
and let $\gamma, \epsilon$ be such that 
\begin{equation}\label{e:gam}
(1-\frac{\gamma}{2\beta}-2\gamma\mu-6\gamma\kappa_n -\gamma \kappa_{n+1}) \geq \epsilon > 0.
\end{equation}
Let $(y_n)_{n\in\NN}$ and $(w_{1,n},\ldots,w_{m,n})_{n\in\NN}$ be 
sequences generated by Algorithm \ref{algo:main}.
Then, 
\begin{enumerate}
\item\label{t:i} $\sum_{k\in\NN}(\|y_n-y_{n+1}\|^2+\sum_{i=1}^m \|w_{i,n}-w_{i,n+1}\|^2) < +\infty$.
\item \label{t:ii} The sequence $(y_n)_{n\in\NN}$ converges weakly to $\overline{x}$ that solves the primal inclusion \eqref{e:pri}.
 For every $i\in\{1,\ldots,m\}$, the sequence $(w_{i,n})_{n\in\NN}$ converges weakly to $\overline{v}_i$. 
Moreover, $(\overline{v}_1,\ldots, \overline{v}_m)$ solves the dual inclusion \eqref{e:dual}.
\item 
\label{t:iii} $\sum_{n\in\NN}(\|By_n-B\overline{x} \|^2+\|B_iw_{i,n}-B_i\overline{v}_i\|^2) <+\infty. 
$
\end{enumerate}
\end{theorem}
\begin{proof}
Let $\KKK$, $\AAA$, $\BB$ and $\SSS$ be defined as in Notation \ref{n:main}.
  It was shown in \cite[Eq. (3.11)]{plc6} and \cite[Eq. (3.12)]{Vu2013}   that 
  \begin{equation}\label{e:3max}
  \text{$\AAA$ is maximally monotone and $\SSS$ is monotone, $\mu$-Lipschitzian, and $\BB$ is $\beta'$-cocoercive.}
  \end{equation}
  Let $n\in\NN$.
  Recall that $\AAA_{\gamma} = \gamma^{-1}(\Id-  J_{\gamma \AAA})$ 
is the Yosida approximation of $\AAA$ of index $\gamma$. 
  In view of Lemma \ref{l:1}, the operator
  \begin{align}\label{e:inverse}
\Id-\gamma \SSS-\gamma\BB\colon\zer(\AAA+\SSS+ \BB) 
\to \zer((\BB +\SSS)J_{\gamma \AAA} +& \AAA _{\gamma} )\;\notag\\
 &\text{  is a bijection with inverse $ J_{\gamma\AAA} $,}
  \end{align}
  where,   using  \cite[Proposition 23.18]{livre1}, 
the resolvent of $\AAA$ is given by
 \begin{align} 
 (\forall \xx = (x,v_1,\ldots, v_m)&\in\KK)(\forall \gamma \in \left]0,+\infty\right[)\; \notag\\
&J_{\gamma \AAA}\xx = \big( J_{\gamma A}(x+\gamma z), J_{\gamma A_{1}}(v_1-\gamma r_1), \ldots 
J_{\gamma A_{m}} (v_m-\gamma r_m)\big).
 \end{align}
  For every $n\in\NN$, define 
  \begin{equation}
 \begin{cases}
\SSS_n \colon \KKK\to\KKK \colon (x,v_1,\ldots, v_m) \mapsto  (Q_nx+\sum_{i=1}^m L_{i}^*v_i, Q_{1,n}v_1 -L_1x, \ldots, Q_{m,n}v_m -L_m x ),\\
 \BB_n \colon \KKK\to\KKK \colon (x,v_1,\ldots, v_m) \mapsto (B_nx, B_{1,n}v_1,\ldots,  B_{m,n}v_m).
 \end{cases}
 \end{equation}
 Then $\SSS_n-\SSS$ and $\BB_n-\BB$ are $\kappa_n$-Lipschitz continuous. Moreover, for $\cc= (c_0,c_1,\ldots,c_m)$ and 
 $\dd = (d_0,d_1,\ldots,d_m)$, 
 \begin{equation}
     \SSS_n\dd =\SSS\dd\; \text{and}\; \BB_n \cc = \BB \cc.
 \end{equation}
 We next set 
\begin{equation}
 \xx_n  = (x_{n}, v_{1,n},\ldots, v_{m,n})\; \text{and}\; 
\yy_n =  (y_{n}, w_{1,n},\ldots, w_{m,n})=J_{\gamma \AAA}\xx_n.
 \end{equation}
 Then the proposed algorithm can be rewritten in the space $\KK$ as follows 
 \begin{equation}
 \label{eq:281}
(\forall n\in\NN)\quad
\begin{array}{l}
\left\lfloor
\begin{array}{l}
\yy_n = J_{\gamma \AAA}\xx_n\\
\xx_{n+1} = \yy_n - \gamma 2\SSS_n \yy_n +\gamma \SSS_n \yy_{n-1} -\gamma \BB_n  \yy_n.
\end{array}
\right.\\[2mm]
\end{array}
\end{equation}
  Let $\xx \in \zer(\AAA +\SSS+\BB)$. Then,
it follows from \eqref{e:inverse} that there exists 
$\yy \in \zer( (\BBB+\SSS)J_{\gamma A} +\AAA_{\gamma})$ 
such that $\xx = J_{\gamma \AAA}\yy $ and $\yy = (\Id -\gamma \SSS-\gamma \BB)\xx$.
 Since $J_{\gamma \AAA}$ is firmly nonexpansive, we obtain 
 \begin{align} \label{eq:286}
2 \|\yy_{n+1}-\xx\|^2&\leq 2\scal{\yy_{n+1}-\xx}{\xx_{n+1}-\yy}\notag\\
 &=2\scal{\yy_{n+1}-\xx}{\yy_n-\xx -\gamma(2\SSS_n\yy_n-\SSS_n\yy_{n-1})-\SSS\xx) -\gamma (\BB_n\yy_n -\BB\xx)}\notag\\
 &= 2\Gamma_{1,n} -2\gamma\Gamma_{2,n} -2\gamma\Gamma_{3,n}-2\gamma\Gamma_{4,n}-2\gamma\Gamma_{5,n}, 
 \end{align} 
 where we set
 \begin{equation} \label{eq:291}
 \begin{cases}
   \Gamma_{1,n} &=\scal{\yy_{n+1}-\xx}{\yy_n-\xx} = \frac12(\|\yy_{n+1}-\xx \|^2 +   \|\yy_{n}-\xx \|^2 -  \|\yy_{n+1}-\yy_n \|^2),\\
   \Gamma_{2,n} &=\scal{\yy_{n+1}-\xx}{\SSS\yy_n-\SSS\yy_{n-1}}= \scal{\yy_{n+1}-\yy_n}{\SSS\yy_n-\SSS\yy_{n-1}}+\scal{\yy_{n}-\xx}{\SSS\yy_n-\SSS\yy_{n-1}},\\
   \Gamma_{3,n} &=\scal{\yy_{n+1}-\xx}{\SSS\yy_{n}-\SSS\xx}= \scal{\yy_{n+1}-\xx}{\SSS\yy_n-\SSS\yy_{n+1}}+\scal{\yy_{n+1}-\xx}{\SSS\yy_{n+1}-\SSS\xx}\\
   \Gamma_{4,n} &=\scal{\yy_{n+1}-\xx}{\BB\yy_{n}-\BB\xx}= \scal{\yy_{n+1}-\yy_n}{\BB\yy_n-\BB\xx}+\scal{\yy_{n}-\xx}{\BB\yy_{n}-\BB\xx}\\
   \Gamma_{5,n} &= \scal{\yy_{n+1}-\xx}{2(\SSS_n-\SSS)\yy_n-(\SSS_n-\SSS)\yy_{n-1} +(\BB_n-\BB)\yy_n}.
 \end{cases}
 \end{equation}
 Since $\SSS_n\dd=\SSS\dd$ and $\SSS_n-\SSS$ is $\kappa_n$-Lipschitz, we have 
 \begin{align}
 |\scal{\yy_{n+1}-\xx}{(\SSS_n-\SSS)\yy_n}|
 &=\scal{\yy_{n+1}-\xx}{(\SSS_n-\SSS)\yy_n -(\SSS_n-\SSS)\dd}\notag\\
 &\leq \|\yy_{n+1}-\xx\|\|(\SSS_n-\SSS)\yy_n -(\SSS_n-\SSS)\dd\|\notag\\
 &\leq 
 \kappa_n\|\yy_{n+1}-\xx\|(\|\yy_n -\xx\|+\|\xx-\dd\|)\notag\\
  &\leq \frac{\kappa_n}{2}\|\yy_{n+1}-\xx\|^2 + \kappa_n\|\yy_n-\xx\|^2 + \kappa_n\|\dd-\xx\|^2, \label{e:tt1}
 \end{align}
 and by the same way
 \begin{align}
 |\scal{\yy_{n+1}-\xx}{(\SSS_n-\SSS)\yy_{n}-(\SSS_n-\SSS)\yy_{n-1}}| \leq \frac{\kappa_n}{2} \|\yy_{n+1}-\xx\|^2 
 +\frac{\kappa_n}{2} \|\yy_n-\yy_{n-1}\|^2. \label{e:tt2}
 \end{align}
  Since $\BB_n\cc=\BB\cc$ and $\BB_n-\BB$ is $\kappa_n$-Lipschitz, we have 
 \begin{align}
 |\scal{\yy_{n+1}-\xx}{(\BB_n-\BB)\yy_n}|
  &\leq \frac{\kappa_n}{2}\|\yy_{n+1}-\xx\|^2 + \kappa_n\|\yy_n-\xx\|^2 + \kappa_n\|\cc-\xx\|^2. \label{e:tt3}
 \end{align}
 Adding \eqref{e:tt1}, \eqref{e:tt2} and \eqref{e:tt3}, we 
 obtain the estimation $\Gamma_{5,n}$, 
 \begin{align}
 2\gamma\Gamma_{5,n} &=  3\gamma\kappa_n\|\yy_{n+1}-\xx\|^2
 +4\gamma \kappa_n\|\yy_n-\xx\|^2 +\gamma\kappa_n \|\yy_n-\yy_{n-1}\|^2+\ee_n\notag\\
 &\leq 10\gamma \kappa_n\|\yy_n-\xx\|^2 +6\gamma\kappa_n \|\yy_n-\yy_{n+1}\|^2+\gamma\kappa_n \|\yy_n-\yy_{n-1}\|^2+\ee_n,
 \end{align}
 where we set $\ee_n= 2\gamma\kappa_n(\|\dd-\xx\|^2+\|\cc-\xx\|^2) $.
 Let us estimate $\Gamma_{2,n}$, $\Gamma_{3,n}$ and $\Gamma_{4,n}$. Since $\SSS$ is $\mu$-Lipschitzian, we obtain
\begin{align}\label{eq:300}
 \Gamma_{2,n} &\geq \scal{\yy_{n}-\xx}{\SSS\yy_n-\SSS\yy_{n-1}} -  \|\yy_{n+1}-\yy_n \| \|\SSS\yy_{n-1}-\SSS\yy_n \|\notag\\
&\geq \scal{\yy_{n}-\xx}{\SSS\yy_n-\SSS\yy_{n-1}} -  \mu\|\yy_{n+1}-\yy_n \| \|\yy_{n-1}-\yy_n \|\notag\\
&\geq \scal{\yy_{n}-\xx}{\SSS\yy_n-\SSS\yy_{n-1}} - (\mu/2) \|\yy_{n+1}-\yy_n \|^2 - (\mu/2) \|\yy_{n-1}-\yy_n \|^2,
\end{align}
where the last inequality follows from the Cauchy-Schwarz inequality. Since $\SSS$ is monotone,  
 we get
  \begin{equation} \label{eq:307}
 \Gamma_{3,n} \geq  \scal{\yy_{n+1}-\xx}{\SSS\yy_n-\SSS\yy_{n+1}}.
\end{equation}
 Set $\beta' = \min_{0\leq i\leq m} \beta_i$ and $\epsilon_1 = \beta'-\beta > 0$. Then,
  using the $\beta$-cocoercivity of $B$, 
\begin{align} \label{eq:311}
 \Gamma_{4,n} & \geq \scal{\yy_{n+1}-\yy_n}{\BB\yy_n-\BB\xx} + \beta' \|\BB\yy_n-\BB\xx \|^2\notag\\
&\geq-\frac{1}{4\beta}\|\yy_{n+1}-\yy_n\|^2 + \epsilon_1 \|\BB\yy_n-\BB\xx \|^2.
\end{align}
 Now, inserting \eqref{eq:291}, \eqref{eq:300}, \eqref{eq:307} and \eqref{eq:311} into \eqref{eq:286},  
we obtain
 \begin{align} \label{eq:317}
 &\|\yy_{n+1}-\xx\|^2 + \gamma(\mu+\kappa_{n+1}) \|\yy_{n+1}-\yy_n \|^2 +2\gamma\scal{\yy_{n+1}-\xx}{\SSS\yy_n-\SSS\yy_{n+1}}\notag\\
&\leq  (1+10\gamma\kappa_n)\|\yy_{n}-\xx\|^2+ \gamma(\kappa_n+\mu) \|\yy_{n-1}-\yy_n \|^2+2\gamma \scal{\yy_{n}-\xx}{\SSS\yy_{n-1}-\SSS\yy_n}\notag\\
&\hspace{1.5cm}-(1-\frac{\gamma}{2\beta}-2\gamma\mu-6\gamma\kappa_n -\gamma \kappa_{n+1})\|\yy_{n+1}-\yy_n\|^2+\ee_n-2\gamma \epsilon_1 \|\BB\yy_n-\BB\xx \|^2.
 \end{align} 
 We next set 
 $t_n=(\kappa_n+\mu) \|\yy_{n-1}-\yy_n \|^2 -\|\SSS\yy_n-\SSS\yy_{n-1}\|^2 \geq 0 $ and
 using
 \begin{equation*}
 2\gamma\scal{\yy_{n+1}-\xx}{\SSS\yy_n-\SSS\yy_{n+1}}
 = \gamma \|\yy_{n+1}-\xx +\SSS\yy_n-\SSS\yy_{n+1} \|^2 
 -\gamma \|\yy_{n+1}-\xx\|^2- \gamma \|\SSS\yy_n-\SSS\yy_{n+1} \|^2,
 \end{equation*}
 we derive from \eqref{eq:317} that 
  \begin{align} \label{eq:317aa}
 &(1-\gamma)\|\yy_{n+1}-\xx\|^2 + \gamma t_{n+1} +\gamma \|\yy_{n+1}-\xx +\SSS\yy_n-\SSS\yy_{n+1} \|^2\notag\\
&\leq  (1-\gamma+10\gamma\kappa_n)\|\yy_{n}-\xx\|^2+ \gamma t_n+ \gamma \|\yy_{n}-\xx +\SSS\yy_{n-1}-\SSS\yy_{n} \|^2 
\notag\\
&\hspace{2.5cm}-(1-\frac{\gamma}{2\beta}-2\gamma\mu-6\gamma\kappa_n -\gamma \kappa_{n+1})\|\yy_{n+1}-\yy_n\|^2+\ee_n-2\gamma \epsilon_1 \|\BB\yy_n-\BB\xx \|^2.
 \end{align} 
 In turn, 
   \begin{align} \label{eq:317ab}
 &\|\yy_{n+1}-\xx\|^2 + \frac{\gamma}{1-\gamma} t_{n+1} +\frac{\gamma}{1-\gamma} \|\yy_{n+1}-\xx +\SSS\yy_n-\SSS\yy_{n+1} \|^2\notag\\
&\leq  (1+ \frac{10\gamma\kappa_n}{1-\gamma})\left[\|\yy_{n}-\xx\|^2+ \frac{\gamma}{1-\gamma} t_n+ \frac{\gamma}{1-\gamma} \|\yy_{n}-\xx +\SSS\yy_{n-1}-\SSS\yy_{n} \|^2\right] 
\notag\\
&\hspace{0.5cm}-\frac{1}{1-\gamma}\left[(1-\frac{\gamma}{2\beta}-2\gamma\mu-6\gamma\kappa_n -\gamma \kappa_{n+1})\|\yy_{n+1}-\yy_n\|^2-\ee_n+2\gamma \epsilon_1 \|\BB\yy_n-\BB\xx \|^2\right].
 \end{align} 
 Since $(\ee_n)_{n\in\NN}$ and $(\kappa_n)_{n\in\NN}$ are  summable sequence, it follows \eqref{eq:317ab} that 
 $(\|\yy_n-\xx\|)_{n\in\NN}$ is a bounded seqeunce and that
 \begin{equation}\label{eq:332}
\sum_{n\in\NN} \|\yy_{n+1}-\yy_n\|^2 < +\infty \; \text{and}\; \sum_{n\in\NN} \|\BB\yy_n-\BB\xx\|^2 < +\infty,
\end{equation}
and that 
\begin{equation}\label{eq:336}
\exists\;  \xi = \lim_{n\to\infty}\bigg(\|\yy_{n+1}-\xx\|^2 + \frac{\gamma}{1-\gamma} t_{n+1} +\frac{\gamma}{1-\gamma} \|\yy_{n+1}-\xx +\SSS\yy_n-\SSS\yy_{n+1} \|^2 \bigg) \in\RR_{+}.
\end{equation}
 Moreover,   
by \eqref{eq:332},  
  \begin{equation}\label{eq:342}
 \lim_{n\to\infty} \big(2\gamma\scal{\yy_{n+1}-\xx}{\SSS\yy_n-\SSS\yy_{n+1}}\big) =\lim_{n\to\infty}t_n =0.
\end{equation} 
Therefore, \eqref{eq:336} is equivalent to 
 \begin{equation}\label{eq:346}
\exists\;  \xi = \lim_{n\to\infty}\|\yy_{n}-\xx\|^2  \in\RR_{+}.
\end{equation}
\ref{t:ii}. We first prove that every sequentially weak cluster point of $(\yy_n)_{n\in\NN}$ is in $\zer(\AAA+\BB+\SSS)$. Suppose
that $(\yy_{k_n})_{n\in\NN}$ is a subsequence of $(\yy_n)_{n\in\NN}$ that converges weakly to $\yy\in\KKK$. By \eqref{eq:346}, we also have
$\yy_{1+k_n}\weakly \yy$. Furthermore, by the definition of $\yy_{1+k_n}$ in \eqref{eq:281}, we have
\begin{equation}\label{eq:352}
\xx_{1+k_n}-\yy_{1+k_n}\in \gamma \AAA \yy_{1+k_n},
\end{equation}
and the definition of $\xx_{n+1}$ in \eqref{eq:281}, 
\begin{align}\label{eq:356}
\xx_{1+k_n} &= \yy_{k_n}-\gamma (2\SSS_n\yy_{k_n}-\SSS_n\yy_{k_n-1}) -\gamma \BB_n\yy_{k_n}\notag\\
&=  \yy_{k_n}-\gamma (2\SSS\yy_{k_n}-\SSS \yy_{k_n-1}) -\gamma \BB\yy_{k_n} + \aaa_n,
\end{align}
where we set 
\begin{equation}
\aaa_n=2\gamma(\SSS\yy_{k_n} - \SSS_n\yy_{k_n}  ) 
+\gamma (\SSS_n\yy_{k_n-1}-\SSS \yy_{k_n-1})+\gamma(\BB\yy_{k_n} -\BB_n\yy_{k_n})\to 0.
\end{equation}
Adding \eqref{eq:352} and \eqref{eq:356} gives
\begin{align}\label{eq:356c}
 \yy_{k_n}-\yy_{1+k_n}-\gamma (\SSS\yy_{k_n}-\SSS\yy_{k_n-1}) -\gamma (\SSS\yy_{k_n}-&\SSS\yy_{1+k_n}) +\aaa_n
-\gamma (\BB\yy_{k_n}-\BB\yy_{1+k_n})\notag\\
 &\in \gamma(\AAA+\SSS+\BB)\yy_{1+k_n}.
\end{align}
Using \eqref{e:3max} and \cite[Corollary 25.5]{livre1}, the sum $\AAA+\SSS+\BBB$ is  maximally monotone. Hence, its graph is closed in $\KKK^{weak}\times\KKK^{strong}$ \cite[Proposition 20.38]{livre1}. Note that, in view of \eqref{eq:342} and the Lipschitzianity of $\SSS$ and $\BB$, we have 
\begin{equation}
 \yy_{k_n}-\yy_{1+k_n}-\gamma \SSS(\yy_{k_n}-\yy_{k_n-1}) -\gamma \SSS(\yy_{k_n}-\yy_{1+k_n}) 
-\gamma (\BB\yy_{k_n}-\BB\yy_{1+k_n})\to 0,
\end{equation}
which, and $\yy_{1+k_n}\weakly \yy$, implies that $\yy \in\zer(\AAA+\SSS+\BB)$. By Opial's result \cite{OPial}, the $(\yy_n)_{n\in\NN}$ 
converges weakly to some point $\overline{\xx} = (\overline{x},\overline{v}_1,\ldots,\overline{v}_m)\in\zer(\AAA+\SSS+\BB)$. Therefore,
$y_n\weakly \overline{x}$ and for every $i\in\{1,\ldots,m\}$, $w_{i,n}\weakly \overline{v}_i$. By \eqref{e:dss}, $ \overline{x}\; \text{solves}\; \eqref{e:pri}\; \text{and}\;   ( \overline{v}_1,\ldots, \overline{v}_m)\;
   \text{solves}\; \eqref{e:dual}$.
\\
\noindent
\ref{t:i} $\&$ \ref{t:iii}. The conclusions follow from \eqref{eq:336}.

\end{proof}
\begin{theorem} \label{t:2} 
Under the setting of Problem \ref{p:primal-dual} and the same conditions stated in Theorem \ref{t:1}.
Let $(y_n)_{n\in\NN}$ and $(w_{1,n},\ldots,w_{m,n})_{n\in\NN}$ be 
sequences generated by Algorithm \ref{algo:main}. 
Then, the following hold.
\begin{enumerate}
 \item  \label{t:2i} If $A+B+Q$ is uniformly monotone at $\overline{x}$,
the sequence $(y_n)_{n\in\NN}$ converges strongly to $\overline{x}$.
 \item \label{t:2ii}  If $A_j+B_j+Q_j$ is uniformly monotone at $\overline{v}_j$, for some $j\in\{1,\ldots,m\}$, then
$(w_{j,n})_{n\in\NN}$ converges strongly to $\overline{v}_j$.
\end{enumerate}
\end{theorem}
\begin{proof}  Let $n\in\NN$ and 
set
\begin{equation}
    p_{n+1}=(x_{n+1}-y_{n+1})/\gamma + B_ny_{n+1}+Q_ny_{n+1} +\sum_{i=1}^m L_{i}^*w_{i,n+1}.
\end{equation} 
Then, 
\begin{align}
p_{n+1}+&(y_{n+1}-y_{n})/\gamma =  (x_{n+1}-y_n)/\gamma+  B_ny_{n+1}+Q_ny_{n+1}  \notag\\
&=-\sum_{i=1}^m L_{i}^*(w_{i,n}-w_{i,n-1}) - (Q_n y_{n}- Q_ny_{n-1} )- (B_ny_{n}-B_ny_{n+1}) - (Q_ny_n-Q_ny_{n+1})\notag\\
&+\sum_{i=1}^mL_i(w_{i,n+1}-w_{i,n}).
\label{e:re1}
\end{align}
Since
$\sum_{k\in\NN}\|\yy_k-\yy_{k+1}\|^2 < +\infty$,
we obtain $y_n-y_{n+1}\to 0$ and
$(\forall i\in\{1,\ldots,m\})\; w_{i,n}-w_{i,n+1}\to 0$.  Hence, it follows from the Lipschitzianity of $Q,B$ and $Q_n-Q, B_n-B$ that
\begin{equation}
    \begin{cases}
    \|Q_ny_n-Q_ny_{n-1}| \leq \|Qy_n-Qy_{n-1}\|+\|(Q_n-Q)y_n-(Q_n-Q)y_{n-1}\| \to 0\notag\\
     \|B_ny_n-B_ny_{n-1}| \leq \|By_n-By_{n-1}\|+\|(B_n-B)y_n-(B_n-B)y_{n-1}\| 
    \end{cases}
\end{equation}
In turn, it follows from  \eqref{e:re1} that 
\begin{equation}
    p_{n+1}\to 0.\label{e:0a}
\end{equation}
We next set 
\begin{equation}
    (\forall i\in \{1,\ldots,n\})\;
    q_{i,n+1}= (v_{i,n+1}-w_{i,n+1})/\gamma 
    + B_{i,n}w_{i,n+1} + Q_{i,n}w_{i,n+1} -L_iy_{n+1}. 
\end{equation}
Then, for each $i\in \{1,\ldots,n\}$, we have 
\begin{align}
 q_{i,n+1}&+(w_{i,n+1} -w_{i,n})/\gamma = (v_{i,n+1}-w_{i,n})/\gamma 
    + B_{i,n}w_{i,n+1} + Q_{i,n}w_{i,n+1} -L_iy_{n+1}\notag\\
   & =- (Q_{i,n} w_{i,n}- Q_{i,n}w_{i,n-1} ) +  L_i(y_{n}-y_{n-1}) -( B_{i,n} w_{i,n}-B_{i,n}w_{i+1,n}) -( Q_{i,n} w_{i,n}-Q_{i,n}w_{i+1,n})\notag\\
   &+ L_i(y_{n}-y_{n+1}).
    \label{e:re2}
\end{align}
Since $y_n-y_{n+1}\to 0$ and
$(\forall i\in\{1,\ldots,m\})\; w_{i,n}-w_{i,n+1}\to 0$, and
$(L_i)_{1\leq i\leq m}$, $B$, $Q$, $B_n-B, Q_n-Q$ 
are Lipschitz continuous, 
we derive from \eqref{e:re2} that 
\begin{equation}
    q_{i,n+1} \to 0.\label{e:0b}
\end{equation}
Now,  by the definition of $(y_n)_{n\in\NN}$, we have 
 \begin{equation}
 \overline{p}_{n+1}= p_{n+1}+ (B+Q)y_{n+1}-(B_n+Q_n)y_{n+1} \in-z+ (A+B+Q)y_{n+1}+\sum_{i=1}^m L_{i}^*w_{i,n+1},\label{e:re1a}
 \end{equation}
 and by the definition of $(v_{i,n})_{n\in\NN}$, 
 \begin{equation}
 \overline{q}_{i,n+1}= q_{i,n+1}+ (B_i+Q_i)w_{i,n+1}-(B_{i,n}+Q_{i,n})w_{i,n+1}\in r_i+ (A_i+B_i+Q_i) w_{i,n+1}-L_iy_{n+1}.\label{e:re2a}
 \end{equation}
 By Condition \ref{c:new}, 
 \begin{align}
     \overline{p}_{n+1} \leq \|p_{n+1}\| + \kappa_{0,n} (\|y_{n+1}-c_0\| +\|y_{n+1}-d_0\|) \to 0,
 \end{align}
 and for every $i\in\{1,\ldots,m\}$,
  \begin{align}
     \overline{q}_{i,n+1} \leq \|q_{n+1}\| + \kappa_{i,n} (\|w_{i,n}-c_{i}\| +\|w_{i,n}-d_i\|) \to 0.
 \end{align}
 We recall that $(y_n,w_{1,n},\ldots,w_{m,n})_{n\in\NN}$ converges weakly to
  $(\overline{y},\overline{v}_1,\ldots,\overline{y}_m) \in \zer(\AAA+\SSS+\BB)$ which verifies 
 \begin{equation}
     -\sum_{i=1}^mL^{*}_i\overline{v}_i \in -z + (A+B+Q)\overline{y}. \label{e:re1b}
 \end{equation}
 and 
 \begin{equation}
     (\forall i\in\{1,\ldots,m\})\;  L_i\overline{y} \in r_i+ (A_i+B_i+Q_i)\overline{v}_i.\label{e:re2b}
 \end{equation}
  (i): Suppose that $A+B+Q$ is uniformly monotone at $\overline{y}$. Then there exists an increasing function $\phi_{A+B+Q}\colon\left[0,\infty\right[\to \left[0,\infty\right]$ that vanishes only at $0$ such that 
 \begin{equation}\label{unif}
\big(\forall u\in (A+B+Q)x\big)\big(\forall (y,v)\in\gra (A+B+Q)\big)
\quad\scal{x-y}{u-v}\geq \phi_{A+B+Q}(\|y-x\|).
\end{equation}
 Therefore, we derive from \eqref{e:re1a} and \eqref{e:re1b} that  
 \begin{equation}
    \scal{\overline{p}_{n+1}-\sum_{i=1}^mL_{i}^*(w_{i,n+1}-\overline{v}_i)}{y_{n+1}-\overline{y}}\geq \phi_{A+B+Q}(\|y_{n+1}-\overline{y}\|). \label{e:re3a}
 \end{equation}
 For every $i\in\{1,\ldots,m\}$, $A_i+B_i+Q_i$ is monotone, it follows from \eqref{e:re2a} and \eqref{e:re2b} that
 \begin{equation}
   (\forall i\in\{1,\ldots,m\})\;  \scal{\overline{q}_{i,n+1}+L_i(y_{n+1}-\overline{y})}{w_{i,n+1}-\overline{v}_i}\geq 0,
 \end{equation}
 which implies that 
 \begin{equation}
     \sum_{i=1}^m\scal{\overline{q}_{i,n+1}+L_i(y_{n+1}-\overline{y})}{w_{i,n+1}-\overline{v}_i}\geq 0.\label{e:re3b}
 \end{equation}
 Adding \eqref{e:re3a} and \eqref{e:re3b}, we obtain 
 \begin{align}
    \phi_{A+B+Q}(\|y_{n+1}-\overline{y}\|) 
    &\leq \scal{\overline{p}_{n+1}}{y_{n+1}-\overline{y}} +
    \sum_{i=1}^m\scal{\overline{q}_{i,n+1}}{w_{i,n+1}-\overline{v}_i}\notag\\
    &\to 0,
 \end{align}
 where the last implication follows from \eqref{e:0a}, \eqref{e:0b} and Theorem \ref{t:1}. Therefore, $y_n\to 0$.
 
 (ii): Suppose that $A_j+B_j+Q_j$ is uniformly monotone at $\overline{v}_j$ for some $j\in\{1,\ldots,m\}$. Then using the same argument, we also have 
 \begin{align}
    \phi_{A_j+B_j+Q_j}(\|w_{j,n+1}-\overline{v}_j\|) 
    &\leq \scal{\overline{p}_{n+1}}{y_{n+1}-\overline{y}} +
    \sum_{i=1}^m\scal{\overline{q}_{i,n+1}}{w_{i,n+1}-\overline{v}_i}\notag\\
    &\to 0,
    \end{align}
    which implies that $w_{j,n}\to 0$.
 \end{proof}
 
 \begin{corollary} \label{coro:1}
 Under the same conditions as in Theorem \ref{t:1}. Let $(x_n)_{n\in\NN}$ and $((v_{1,n},\ldots, v_{m,n}))_{n\in\NN}$
 be sequences generated by Algorithm \ref{algo:main}. Then the following hold.
 \begin{enumerate}
 \item If $A+B+Q$ is uniformly monotone at $\overline{x}$, then 
 \begin{equation}\label{e:Q1}
 x_n\to \overline{y}= \overline{x} -\gamma \sum_{i=1}^m L_{i}^*\overline{v}_i -\gamma Q\overline{x} -\gamma B\overline{x}.
 \end{equation}
 \item If there exists $j\in\{1,\ldots,m\}$ such that $A_j+B_j+Q_j$ is uniformly monotone at $\overline{v}_j$,
 for some $j\in\{1,\ldots,m\}$,
 then 
  \begin{equation}\label{e:Qj}
 v_{j,n}\to \overline{w}_j= \overline{x} +\gamma L_{j}^*\overline{x} -\gamma Q_j\overline{v}_j -\gamma B_j\overline{v}_j.
 \end{equation}
 \end{enumerate}
 \end{corollary}
\begin{corollary} Under the same condition as in Lemma \ref{l:M-T}. Let $(x_n,y_n)_{n\in\NN}$ be gernerated by \eqref{algo:primafbl}. Suppose that $A+B+Q$ is uniformy monotone at 
$\overline{x}$. Then $y_n\to x\in \zer(A+B+Q)$ and $x_n\to \overline{x}-\gamma (Q+B)\overline{x}$.
\end{corollary}

 \begin{remark} Here are some remarks.
 \begin{enumerate}
     \item In the case when $\kappa_n\equiv 0$, as we showed in the proof of Theorem \ref{t:1} that in the space $\KKK$, Algorithm \ref{algo:main} reduces to 
 \eqref{eq:281}, with $\SSS_n\equiv \SSS, \BB_n\equiv \BB$,  which has the backward reflected forward structure. Therefore, Algorithm \ref{algo:main}
 is a primal-dual backward reflected forward splitting method. Alternative framework can be found  in 
 \cite{Bot2013,Buicomb,siop2, plc6,Combettes13,Icip14,Condat2013,sva2,Latafat2018,Pesquet15, Bang13,Bang14,Vu2013}.
 \item Whence $\kappa_n\not=0$, to the best of our knowledge, this paper was first used this type of approximation for the Lipschitzian operators.
  \item The uniform monotonicity of  $A+B+Q$ is satisfied, for instance, when either $A$ or $B$ or $Q$ is uniformly monotone. 
 \item Algorithm \ref{algo:main} is different from the methods in \cite{plc6} and \cite{Vu2013} where the strong convergence is proved under the uniform monotonicity of either $A$ or $Q$ or $B$. 
 \item By Lemma \ref{l:1}, the point $\ww= (\overline{y},\overline{w}_1,\ldots,\overline{w}_m)$ defined by \eqref{e:Q1} and \eqref{e:Qj} is a solution to the inclusion:
 \begin{equation}
     0\in (\BB+\SSS) J_{\gamma \AAA}\ww +\AAA_{\gamma}\ww.
 \end{equation}
 This is an elegant property of Algorithm \ref{algo:main}; see \cite{Attouch18} for detailed comments in the case of the backward-forward splitting.
 \end{enumerate}
 \end{remark}
 \section{Application to structured minimization problems}
 \label{s:Application}
 We provide a further application to the following structure minimization problem which was investigated in \cite{plc6,Vu2013}. 
 This structured primal-dual framework covers a widely class of convex
optimization problems and it has found many applications to image processing, machine learning
\cite{plc6,Pesquet15,Duchi09,Quyen14}.
 \begin{example}\label{App1}
Let $\HH$ be a real Hilbert space and $z\in\HH$, let $f\in\Gamma_0(\HH)$, and let $h\colon\HH\to\RR$ be convex and
differentiable with a $\beta_{0}^{-1}$-Lipschitzian gradient for 
some $\beta_0\in\left]0,+\infty\right[$.
Let $m$ be a strictly positive integer.
For every $i \in \{1,\ldots, m\}$, let $\GG_i$ be a real Hilbert space and $r_i\in\GG_i$,
 let $g_i \in \Gamma_0(\GG_i)$, 
let $\ell_i \in \Gamma_0(\GG_i)$ be $\beta_i$-strongly convex, for some
$\beta_i\in\left]0,+\infty\right[$. For every $i \in \{1,\ldots, m\}$, let  $L_i\colon\HH\to\GG_i$ is a nonzero bounded linear operator.
Consider the primal problem
\begin{equation}\label{e:pri1}
\underset{x\in\HH}{\text{minimize}} \;
f(x)+\sum_{i=1}^m(g_i\;\vuo\;\ell_i)(L_ix-r_i) + h(x) -\scal{z}{x},
\end{equation}
and the dual problem  
\begin{equation}\label{e:dual1}
 \underset{v_1\in\GG_1,\ldots,v_m\in\GG_m}{\text{minimize}} \;
(f^*\;\vuo\; h^*)\bigg(z-\sum_{i=1}^mL_{i}^*v_i\bigg) +
\sum_{i=1}^m \big(g^{*}_i(v_i)+\ell^{*}_i(v_i) +\scal{v_i}{r_i}\big).
\end{equation}
We denote by $\mathcal{P}_1$ and $\mathcal{D}_1$ the sets of 
solutions to~\eqref{e:pri1} and~\eqref{e:dual1}, respectively.
\end{example}
\begin{corollary}\label{probex6}
In Example~\ref{App1}, suppose that 
\begin{equation}\label{IVe:ranex6}
z\in\ran\bigg( \partial f + 
\sum_{i=1}^mL^{*}_i
\big((\partial g_i\;\vuo\;\partial\ell_i)(L_i\cdot-r_i)\big)
+\nabla h\bigg).
\end{equation}
For every $n\in\NN$, let $B_n\colon\HH\to\HH$ and
for every $i\in\{1,\ldots,m\}$, let $B_i\colon\GG_i\to\GG_i$ be such that:
\begin{enumerate}
\item 
The operators $B_n-\nabla h$ are $\kappa_{0,n}$-Lipschitz continuous with $\sum_{n\in\NN}\kappa_{0,n} <+\infty$ and 
\begin{equation}
    (\exists c_0\in \HH)(\forall n\in\NN) \; B_nc_0= \nabla h(c_0)\; 
    \end{equation}
\item For every $i\in \{1,\ldots,m\}$, 
the operators $B_{i,n}-\nabla\ell^{*}_i$ are $\kappa_{i,n}$-Lipschitz continuous with $\sum_{n\in\NN}\kappa_{i,n} <+\infty$, and 
\begin{equation}
    (\forall i\in\{1,\ldots,m\})(\exists c_{i},\in \GG_{i})(\forall n\in\NN) \; B_{i,n}c_{i}= \nabla\ell^{*}_i(c_i).
\end{equation}
\end{enumerate}
Let $\beta$, $\mu$ and $(\kappa_n)_{n\in\NN}$ be such that
\begin{equation}\label{eq:310c}
0 < \beta <  \beta'=\min_{0\leq i\leq m} \beta_i\; \text{and}\; 
\mu =\sqrt{\sum_{i=1}^m\|L_i\|^2}, (\forall n\in\NN)\; \kappa_n=\sqrt{\sum_{i=0}^m\kappa_{i,n}^2 },
\end{equation}
and let $\gamma, \epsilon$ be such that 
\begin{equation}\label{e:gam}
(1-\frac{\gamma}{2\beta}-2\gamma\mu-6\gamma\kappa_n -\gamma \kappa_{n+1}) \geq \epsilon > 0.
\end{equation}
Let $x_{-1}\in\HH$, $y_{-1} = \prox_{\gamma f}x_{-1} $, 
and  for every $i\in\{1,\ldots,m\}$, let $v_{i,-1} \in\GG_i$ and $w_{i,-1}= \prox_{\gamma f}v_{i,-1}$.
Let $x_{0}\in\HH$ and  for every $i\in\{1,\ldots,m\}$, let $v_{i,0} \in\GG_i$.
Let $(y_n)_{n\in\NN}$ and $(w_{1,n},\ldots,w_{m,n})_{n\in\NN}$ 
be sequences generated by  the following routine
\begin{equation}\label{IVe:eqalgoex6}
(\forall n\in\NN)\quad
\begin{array}{l}
\left\lfloor
\begin{array}{l}
y_n = \prox_{\gamma f} (x_n+\gamma z)\\
\operatorname{For \; \text{$i =1,\ldots,m$}}\\
\begin{array}{l}
\left\lfloor
\begin{array}{l}
w_{i,n} = \prox_{\gamma g^{*}_i} (v_{i,n}-\gamma r_i)\\
\end{array}
\right.\\[2mm]
\end{array}\\
x_{n+1} = y_{n}- \gamma \sum_{i=1}^m L_{i}^*(2w_{i,n}-w_{i,n-1}) -\gamma B_n(y_{n})\\
\operatorname{For \; \text{$i =1,\ldots,m$}}\\
\begin{array}{l}
\left\lfloor
\begin{array}{l}
v_{i,n+1} = w_{i,n}  + \gamma L_i(2y_{n}-y_{n-1}) -\gamma B_{i,n}(w_{i,n}).\\
\end{array}
\right.\\[2mm]
\end{array}\\
\end{array}
\right.\\[2mm]
\end{array}
\end{equation}
Then the following hold for some $\overline{x}\in\mathcal{P}_1$ and 
$(\overline{v}_{1},\ldots,\overline{v}_{m})\in\mathcal{D}_1$.
\begin{enumerate}
\item\label{7:10:06}
Suppose that $f+h$ is uniformly convex at $\overline{x}$. 
Then $y_n \to \overline{x}$.
\item\label{7:10:07}
Suppose that $g_{i}^*+\ell_{j}^{*}$ is uniformly convex at $\overline{v}_{j}$,
for some $j\in\{1,\ldots,m\}$. Then $w_{j,n}\to\overline{v}_{j}$.
\end{enumerate}
\end{corollary}
\begin{proof}
Set 
\begin{equation}\label{eq:663}
A = \partial f, B=\nabla h, Q=0\;
\text{and}\; (\forall i \in\{1,\ldots,m\})\; Q_i=0,
B_{i}=\nabla \ell^{*}_i, A_i=\partial g^{*}_i,  
\end{equation}
Since $\nabla h$ is $\beta_{0}^{-1}$-Lipschitz continuous, 
by the Baillon-Haddad Theorem~\cite[Corollary 18.17]{livre1}, 
$B$ is $\beta_0$-cocoercive. Since $f\in\Gamma_0(\HH)$, $A$ is maximally monotone. 
Moreover since, 
for every $i\in\{1,\ldots,m\}$, $\ell_{i}$ is $\beta_i$-strongly 
convex, $\partial \ell_{i}$ is $\beta_{i}$-strongly monotone and hence
$B_i$ is $\beta_i$-cocoercive. Since $g_i\in\Gamma_0(\HH)$, $A_i$ is maximally monotone.
Furthermore
\begin{equation}\label{eq:676}
(\forall i \in\{1,\ldots,m\}) \; \partial g_i\;\vuo\;\partial\ell_i= (\partial g^{*}_i+\partial\ell^{*}_i)^{-1} = (A_i+B_i+Q_i)^{-1},
\end{equation}
which shows that the condition \eqref{IVe:ranex6} implies that \eqref{c:1} is satisfied. 
Therefore,  every condition of Problem \ref{p:primal-dual} is satisfied.
Moreover, $J_{\gamma A}=\prox_{\gamma  f}$, and for every $i \in\{1,\ldots,m\}$,   
$J_{\gamma  A_i}=\prox_{\gamma g_{i}^*}$. Hence, \eqref{IVe:eqalgoex6}
is an instance of \eqref{algo:primal-dual}. Since, for every $i \in\{1,\ldots,m\}$, $Q_i=0$, 
\eqref{eq:310c} shows that \eqref{eq:310} is satisfied. 
 Hence, by applying Theorem~\ref{t:2}\ref{t:2i}\&\ref{t:2ii}, 
  we obtain that, if $f+h$ is uniformly monotone at $\overline{x}$,
the sequence $(y_n)_{n\in\NN}$ converges strongly to some
$\overline{x}\in\mathcal{P}_1$, and if $g^*{_j}+\ell_{j}^*$ is uniformly monotone at 
$\overline{v}_j$
 the sequence $(w_{j,n})_{n\in\NN}$ 
converges strongly to $\overline{v}_j$.
\end{proof}
\begin{remark} The algorithm \eqref{IVe:eqalgoex6} is an instance of \eqref{algo:primal-dual} and the weak convergence of the iteration is obtained by applying Theorem \ref{t:1}. When $m=1$, a stochastic version of \eqref{IVe:eqalgoex6} is recently investigated in \cite{VD1} where the further connection to exiting works 
\cite{sva2, Malitsky18b,Tseng00} are provided.
\end{remark}
\begin{remark} In the case, when $(\forall n\in\NN)\; B_n= \nabla h$ and $(\forall i\in\{1,\ldots,n\})\; B_{i,n}= \nabla\ell^{*}_{i,n}$, the algorithm \eqref{IVe:eqalgoex6} and the method in \cite{Vu2013}  use only one call of $B$, $(L_i)_{1\leq i\leq m}$  and $\nabla\ell_{i,n}^*$,  $(L^{*}_i)_{1\leq i\leq m}$ per one iteration. 
Whiles, the method in \cite{plc6} uses two calls  of
 $B$, $(L_i)_{1\leq i\leq m}$  and $\nabla\ell_{i,n}^*$,  $(L^{*}_i)_{1\leq i\leq m}$  per iteration. Hence, the algorithm \eqref{IVe:eqalgoex6} shares the same computational cost as the one in \cite{Vu2013} and their computational cost is cheapter than that of the method in \cite{plc6}.
    
\end{remark}
 
\end{document}